\documentclass[12pt]{amsart}
\usepackage{bbm}
\usepackage[pdftex]{graphicx}
\usepackage{amscd, amssymb, dsfont, upgreek, mathrsfs}
\input{xy}
\xyoption{all}

\newtheorem{Thm}{Theorem}

\newtheorem{Prop}[Thm]{Proposition}
\newtheorem{Def}[Thm]{Definition}
\newtheorem{Def/Thm}[Thm]{Definition/Theorem}

\newtheorem{Lemma}[Thm]{Lemma}

\theoremstyle{remark}
\newtheorem{Rmk}[Thm]{Remark}

\newcommand{\ti }{\times}
\newcommand{\ot }{\otimes}

\newcommand{\lann}{\langle\langle}
\newcommand{\rann}{\rangle\rangle}

\newcommand{\G}{{\bf G}}

\newcommand{\NN}{{\mathbb N}}

\newcommand{\PP }{{\mathbb P}}

\newcommand{\QQ }{{\mathbb Q}}
\newcommand{\CC }{{\mathbb C}}
\newcommand{\ZZ }{{\mathbb Z}}

\newcommand{\vir}{\mathrm{vir}}

\newcommand{\T}{{\mathsf{T}}}

\newcommand{\re}{\mathrm{e}}

\newcommand{\lan}{\langle}
\newcommand{\ran}{\rangle}
\newcommand{\lla}{\langle\!\langle}
\newcommand{\rra}{\rangle\!\rangle}

\newcommand{\pP}{\mathsf{P}}

\newcommand{\ppl}{{\mathsf{P}}\left[}
\newcommand{\ppr}{\right]}

\newcommand{\ret}{\re ^{\T}}

\begin{document}

\title[Vanishing of Gromov-Witten invariants of product of $\PP^1$]
{Vanishing of Gromov-Witten invariants of product of $\PP^1$}

\author{Hyenho Lho}
\address{Department of Mathematics, ETH Z\"urich}
\email {hyenho.lho@math.ethz.ch}

\begin{abstract} 
We prove some vanishing conditions on the Gromov-Witten invarists of $\underbrace{\PP^1 \ti \PP^1 \ti ... \ti \PP^1}_{n}$.

 \end{abstract}

\maketitle

\setcounter{tocdepth}{1} 
\tableofcontents

\setcounter{section}{-1}

\section{Introduction}

\subsection{Gromov-Witten invariants}

Denote by $\PP[n]$ the product of $n$ projective lines:

\begin{align*}
 \PP[n]:= \underbrace{\PP^1 \ti \PP^1 \ti ... \ti \PP^1}_{n}.
\end{align*}

\noindent Denote by 

\begin{align*}
 \overline{M}_{g,m}(\PP[n],\mathsf{d})
\end{align*}
the moduli space of stable maps to $\PP[n]$ with degree $\mathsf{d}=(d_1,d_2,...,d_n)$. It has canonical virtual fundamental class $[\overline{M}_{g,m}(\PP[n],\mathsf{d})]^{vir}$ of dimension $(1-g)(n-3)+m+2\sum_i d_i$. Let $\pi$ be the morphism to the moduli space of stable curves determined by the domain,

\begin{align*}
    \pi : \overline{M}_{g,n}(\PP[n],\mathsf{d}) \rightarrow \overline{M}_{g,n}.
\end{align*}

\noindent For $\mathsf{c}_k=(c_{1,k},c_{2,k},\dots,c_{n,k})\in \ZZ^n_2$, Gromov-Witten invariants are defined by

\begin{align*}
GW_{g,m,\mathsf{d}}^{\PP[n]}[\tau_{k_1}\mathsf{c}_1,\dots,\tau_{k_m}\mathsf{c}_m] :=\int_{[\overline{M}_{g,m}(\PP[n],\mathsf{d})]^{vir}}\prod^{m}_{i=1}\pi^*(\psi_{i})^{k_i} \text{ev}_{i}^*(\ot_{j=1}^n H_j^{c_{j,i}})\,. 
\end{align*}

Gromov-Witten invariants have been studied more than 20 years, see \cite{CKatz},\cite{FP} for an introduction to the subject.
In this paper, we prove the following vanishing conditions on Gromov-Witten invariants of $\PP[n]$.

\begin{Thm}\label{Main} We have
$$GW_{g,n,\mathsf{d}}^{\PP[n]}[\tau_{k_1}\mathsf{c}_1,\dots,\tau_{k_m}\mathsf{c}_m] = 0,$$ if the following conditions hold.

\begin{enumerate}
    \item[(i)] $(g-1+\sum_{k=1}^{m}c_{ik})$ is odd for $1 \le i \le n$,
    \item[(ii)] $3g-3+m-\sum_{i=1}^m k_i <n.$
\end{enumerate}

\end{Thm}



\
\subsection{Quasimap invariants and wall-crossing formula}

Let $\G:=(\CC^*)^n$ act on $(\CC^2)^n$ by the standard diagonal action componentwisly on each component $\CC^2$ so that its associated GIT quotient
\begin{align*}
 (\CC^2)^n/\!\!/\G = \PP[n].
\end{align*}

\noindent With this set-up, Ciocan-Fontanine and Kim defined the quasimap moduli space
\begin{align*}
 Q_{g,m}(\PP[n],\mathsf{d})
\end{align*}
with the canonical virtual fundamental class $[Q_{g,k}(\PP[n],\mathsf{d})]^{vir}$. See \cite{CK,CKM,MOP} for an introduction. For $\mathsf{c}_k=(c_{1,k},c_{2,k},\dots,c_{n,k})\in \ZZ^n_2$, we define quasimap invariants of $\PP[n]$ by
\begin{align*}
QI_{g,m,\mathsf{d}}^{\PP[n]}[\tau_{k_1}\mathsf{c}_1,\dots,\tau_{k_m}\mathsf{c}_m] :=\int_{[Q_{g,m}(\PP[n],\mathsf{d})]^{vir}}\prod^{m}_{i=1}\pi^*(\psi_{i})^{k_i} \text{ev}_{i}^*(\ot_{j=1}^n H_j^{c_{j,i}})\,. 
\end{align*}

In \cite{CKg}, the authors studied the relationship between quasimap invariants and Gromov-Witten invariants. By applying the general theorem to $\PP[n]$, we have the following proposition.
\begin{Prop}\em{(\cite{CKg})}\label{Wall} For $\mathsf{c}_k=(c_{1,k},c_{2,k},\dots,c_{n,k})\in \ZZ^n_2$, we have
 \begin{align*}
GW_{g,m,\mathsf{d}}^{\PP[n]}[\tau_{k_1}\mathsf{c}_1,\dots,\tau_{k_m}\mathsf{c}_m]=QI_{g,m,\mathsf{d}}^{\PP[n]}[\tau_{k_1}\mathsf{c}_1,\dots,\tau_{k_m}\mathsf{c}_m] 
\end{align*}
\end{Prop}

\
\

\noindent i.e. quasimap invariants and Gromov-Witten invariants are same for $\PP[n]$. In section 3, we prove the following theorem which is equivalent to Theorem \ref{Main} by Proposition \ref{Wall}.
\begin{Thm}\label{Second}
We have
$$QI_{g,n,\mathsf{d}}^{\PP[n]}[\tau_{k_1}\mathsf{c}_1,\dots,\tau_{k_m}\mathsf{c}_m] = 0,$$ if the following conditions hold.

\begin{enumerate}
    \item[(i)] $(g-1+\sum_{k=1}^{m}c_{ik})$ is odd for $1 \le i \le n$,
    \item[(ii)] $3g-3+m-\sum_{i=1}^m k_i <n.$
\end{enumerate}
\end{Thm}

To prove Theorem \ref{Second}, we apply the localization strategy introduced by Givental \cite{Elliptic,SS,Book} for Gromov-Witten thoery to the quasimap theory of $\PP[n]$.  

\subsection{Acknowledgments} 
I am very grateful to  I.~Ciocan-Fontanine, B.~Kim, R.~Pandharipande and J.~Shen
for discussions over the years 
about the moduli space of quasimaps and the Gromov-Witten invariants. I was supported by the grant ERC-2012-AdG-320368-MCSK. 






\

\section{Preliminaries}

\subsection{$\mathsf{T}$-equivariant theory}
Let $\mathsf{T}:=(\CC^*)^{2n}$ act on $\PP[n]$ standardly on each components. Denote by $$\boldsymbol{\lambda}=(\lambda_1,\overline{\lambda}_1,\lambda_2,\overline{\lambda}_2,...,\lambda_n,\overline{\lambda}_n)$$ $\mathsf{T}$-equivariant parameters. We will use following specializations throughout the paper;
\begin{align*}
 \lambda_k+\overline{\lambda}_k=0 \, , \  \text{for} \, 1 \le k \le n.
\end{align*}

First we set the notation for the cohomology basis and its dual basis. Let $\{p_i\}$ be the set of $\mathsf{T}$-fixed points of $\PP[n]$ and let $\phi_i$ be the basis of $H^*_\mathsf{T}(\PP[n])$ defined by satisfying followings:
\begin{align}\label{Ebasis}
 \phi_i |_{p_j} = \left\{ \begin{array}{rl} 1 & \text{if } i=j  \\
                                    0 & \text{if } i\ne j  \, . \end{array}\right.
\end{align}
 
\noindent Let $\phi ^i$ be the dual basis with respect to the $\T$-equivariant Poincar\'e pairing, i.e.,
\[ \int _Y \phi _i \phi ^j  = \left\{ \begin{array}{rl} 1 & \text{if } i=j  \\
                                    0 & \text{if } i\ne j   \, ,\end{array}\right.  \]

\noindent For each $\mathsf{T}$-fixed point $p_i\in\PP[n]$, Let
 $$e_i=e(T_{p_i}(\PP[n])) $$
 be the equivariant Euler class of the tangent space of $\PP[n]$ at $p_i$.




The action of $\mathsf{T}$ on $\PP[n]$ lifts to  $Q_{g,0}(\PP[n],\mathsf{d})$. The localization formula of \cite{GP} applied to the virtual fundamental class $[Q_{g,0}(\PP[n],\mathsf{d})]^{\text{vir}}$ will play a fundamental fole in our paper. The $\mathsf{T}$-fixed loci are represented in terms of dual graphs, and the contributions of the $\mathsf{T}$-fixed loci are given by tautologicla classes. The formulas are standard, see \cite{KL,MOP}.

\

\subsection{genus 0 invariants}

In this section, we review the genus zero theory. Integrating along the virtual fundamental class $$[Q_{0, k}(\PP[n], \mathsf{d})]^{\vir}$$ we define correlators
$\lan ... \ran _{0, k, \mathsf{d}}^{0+}$ as follows. For $\gamma _i \in H^*_\T(\PP[n])\ot \QQ (\boldsymbol{\lambda} )$,
\[ \lan \gamma _1\psi  ^{a_1} , ..., \gamma _k\psi  ^{a_k} \ran _{0, k, \mathsf{d}}^{0+} = 
\int _{[Q_{0, k}(\PP[n],\mathsf{d})]^{\vir}} \prod _i ev_i^*(\gamma _i)\psi _i ^{a_i} ,\]
where $\psi _i$ is the psi-class associated to the $i$-th marking and $ev_i$ is the 
$i$-th evaluation map.

Let $Q_{g,k}(\PP[n],\mathsf{d})^{\T, p_i}$ be the fixed loci of $Q_{g,k}(\PP[n],\mathsf{d})$ whose elements have domain components only over $p_i$.  Integrating along the localized cycle class

\begin{align*}
 \frac{[Q_{0, k} (\mathsf{d}) ^{\T, p_i}]^{\vir}}{\ret (N^{\vir}_{Q_{0, k} (\PP[n],\mathsf{d}) ^{\T, p_i}/
Q_{0, k}(\PP[n],\mathsf{d})} )}
\end{align*}
we also define local correlators $\lan ... \ran _{0, k, \beta}^{0+, p_i}$ and  $\lla ... \rra _{0, k, \beta}^{0+, p_i}$ as follows:
\begin{align*}  & \lan \gamma _1\psi  ^{a_1} , ..., \gamma _k\psi  ^{a_k} \ran _{0, k, \mathsf{d}}^{0+, p_i}  =
\int _{\frac{[Q_{0, k} (\PP[n],\mathsf{d}) ^{\T, p_i}]^{\vir}}{\ret (N^{\vir}_{Q_{0, k} (\PP[n],\mathsf{d}) ^{\T, p_i}/
Q_{0, k}(\PP[n],\mathsf{d})} )}} \prod _i ev_i^*(\gamma _i)\psi _i ^{a_i} \ \ ; \\
&  \lla \gamma _1\psi  ^{a_1} , ..., \gamma _k\psi  ^{a_k} \rra _{0, k}^{0+, p_i} \\
& = \sum _{m, \mathsf{d}} \frac{\mathsf{q}^{\mathsf{d}}}{m!}
 \lan    \gamma _1\psi  ^{a_1} , ..., \gamma _k\psi  ^{a_k} , \mathsf{t}, ..., \mathsf{t}  \ran_{0, k+m, \beta}^{0+, p_i},
      \end{align*} 
      where $\mathsf{q}=(q_1,\dots,q_n)$ are formal Novikov variables and $\mathsf{t}=\sum_{i}t_i H_i \in H^*_\mathsf{T}(\PP[n])$ where $H_i$ is the pull back of hyperplane class in $i$-th $\PP^1$.

\

\subsection{insertions of 0+ weighted marking}
For the explicit calculations of various genus $0$ invariants, we need to introduce the notion of 0+ weighted marking. We briefly recall the definitions from \cite{BigI}.

Denote by

\begin{align*}
  Q^{0+,0+}_{g,k|m}(\PP[n],\mathsf{d})
\end{align*}

\noindent the moduli space of genus $g$ (resp. genus zero), degree $\mathsf{d}$ stable quasimaps to $\PP[n]$ with ordinary $k$ pointed markings and infinitesimally weighted $m$ pointed markings, see \cite{BigI} for more explanations.

Denote by

\begin{align*}
  Q^{0+,0+}_{g,k|m}(\PP[n],\mathsf{d})^{T,p_i} 
\end{align*}

\noindent the T-fixed part of $Q^{0+,0+}_{g,k|m}(\PP[n],\mathsf{d})$, whose domain components of universal curves are only over $p_i$.

\

Let $\tilde{H}_i \in H^*([(\CC^2)^n/\G])$ be a lift of $H_i \in H^*(\PP[n])$, i.e. , $\tilde{H}_i|_{\PP[n]}=H_i$. For $\gamma_i \in H^*_{\T} (\PP[n])\ot \QQ (\boldsymbol{\lambda} )$, $\mathsf{t} = \sum_i t_i \tilde{H}_i, \delta _j \in H^*_{\T} ([(\CC^2)^n/\G], \QQ )$, denote

\begin{align*}  & \lan \gamma _1\psi  ^{a_1} , ..., \gamma _k\psi  ^{a_k} ;  \delta _1, ..., \delta _m \ran _{0, k|m, \mathsf{d}}^{0+, 0+}  \\ & =
\int _{[Q^{0+, 0+}_{0, k|m} (\PP[n],\mathsf{d}) ^{\T}]^{\vir}} \prod _i ev_i^*(\gamma _i)\psi _i ^{a_i} \prod _j \hat{ev}_j ^* (\delta _j) \ \ ; \\
&  \lla \gamma _1\psi  ^{a_1} , ..., \gamma _k\psi  ^{a_k} \rra _{0, k}^{0+, 0+} \\
& = \sum _{m, \beta} \frac{\mathsf{q}^{\mathsf{d}}}{m!}
 \lan    \gamma _1\psi  ^{a_1} , ..., \gamma _k\psi  ^{a_k} ; \mathsf{t}, ..., \mathsf{t}  \ran_{0, k|m, \mathsf{d}}^{0+, 0+} \ \  ; \\ 
 & \lan \gamma _1\psi  ^{a_1} , ..., \gamma _k\psi  ^{a_k} ;  \delta _1, ..., \delta _m \ran _{0, k|m, \mathsf{d}}^{0+, 0+, p_i}  \\ & =
\int _{\frac{[Q^{0+, 0+}_{0, k|m} (\PP[n], \mathsf{d}) ^{\T, p_i}]^{\vir}}{\ret (N^{\vir}_{Q^{0+, 0+}_{0, k|m} (\PP[n], \mathsf{d}) ^{\T, p_i}/
Q^{0+, 0+}_{0, k|m}(\PP[n], \mathsf{d})} )}} \prod _i ev_i^*(\gamma _i)\psi _i ^{a_i} \prod _j \hat{ev}_j ^* (\delta _j) \ \  ; \\
&  \lla \gamma _1\psi  ^{a_1} , ..., \gamma _k\psi  ^{a_k} \rra _{0, k}^{0+, 0+, p_i} \\
& = \sum _{m, \mathsf{d}} \frac{\mathsf{q}^{\mathsf{d}}}{m!}
 \lan    \gamma _1\psi  ^{a_1} , ..., \gamma _k\psi  ^{a_k} ; \mathsf{t}, ..., \mathsf{t}  \ran_{0, k|m, \mathsf{d}}^{0+, 0+, p_i} \ \ ,
\end{align*}
where $\hat{ev}_j$ is the evaluation map to $[(\CC^2)^n/\mathbf{G}]$ at the $j$-th infinitesimally weighted  marking.






Let $\{ \phi_i \}$ be the equivariant basis of $H^*_\mathsf{T}(\PP[n])$ satisfying \eqref{Ebasis}. Let us define $\mathds{U}_i$, $\mathds{S}$, $\mathds{V}_{ij}$ by

\begin{align*}
\mathds{U}^{\PP[n]}_{p_i} & : = \lla 1,1 \rra ^{0+,0+,i}_{0,2}\in \QQ[[\mathsf{q},\mathsf{t}]](\boldsymbol{\lambda})\  \ ;   \\
\mathds{S}^{\PP[n]}(\gamma) & : =  \sum _{i, \beta}  \phi ^i \mathsf{q}^\mathsf{d} \lla \frac{\phi _i}{z-\psi} , \gamma \rra _{0, 2, \mathsf{d}}^{0+,0+}\in H^*_{\mathsf{T}}(\PP[n])[[z^{-1},\mathsf{q},\mathsf{t}]]\ot\QQ(\boldsymbol{\lambda})\ \ ;  \\
\mathds{V}^{\PP[n]}_{ij} (x, y)  & := \sum_{\mathsf{d} } \mathsf{q}^\mathsf{d}  
\lla \frac{\phi _i}{x- \psi } ,  \frac{\phi _j}{y - \psi } \rra _{0, 2, \mathsf{d}}^{0+,0+}\in \QQ[[x^{-1},y^{-1},\mathsf{q},\mathsf{t}]](\boldsymbol{\lambda})  \ \ ; \\
 \end{align*}

We recall the following WDVV equation from \cite{CKg},

\begin{align}\label{WDVV}
 e_i\mathds{V}_{ij}(x,y)e_j=\frac{\sum_k\mathds{S}(\phi_k)|_{p_i, z=x} \mathds{S}(\phi^k)|_{p_j, z=y}}{x+y}
\end{align}

To study the properties of $\mathds{S}$ and $\mathds{V}_{ij}$, we need to recall infinitesimal $I$-function defined in \cite{BigI}. Denote by $\mathds{I}$ the infinitesimal $I$-function $\mathds{J}^{0+,0+}$ in \cite{BigI}. Note that $\mathds{S}$ coincide with the definition of infinitesimal $S$-operator $\mathds{S}^{0+,0+}$ in \cite{BigI}. From (5.1.3) in \cite{BigI}, we know explicit form of $\mathds{I}$ of $\PP[n]$

\begin{align}\label{I}
 \mathds{I}^{\PP[n]}  &= \prod_{i=1}^n e^{t_i(H_i+d_iz)/z}\frac{q_i^{d_i}}{\prod_{k=1}^{d_i} (H_i-\lambda_i+kz)(H_i-\overline{\lambda}_i+kz)}\\
 \nonumber &\in \mathsf{R}[[z,z^{-1},\mathsf{q},\mathsf{t}]](\mathsf{\lambda}).
\end{align}


To find the explicit forms of $\mathds{S}$-operators, we recall following proposition from \cite{KL}.
 
\begin{Prop}\em{(Proposition 2.4 in \cite{KL})}\label{S}  There are unique coefficients $ a_i (z, \mathsf{q}) \in \QQ (\lambda)[z][[\mathsf{q}]]$ making
 \[ \sum _{i} a_i (z, \mathsf{q}) \partial _{\phi_i} \mathds{I}^{\PP[n]}  = \gamma + O(1/z) .\] Furthermore
 LHS coincides with  $\mathds{S}^{\PP[n]}(\gamma)$.
\end{Prop}

\

\subsection{Ordered graphs.} Let the genus $g$ and the number of markings $n$ for the moduli space be in the stable range
$$2g-2+n > 0. $$
We can organize the $\mathsf{T}$-fixed loci of $Q_{g,n}(\PP[n],\mathsf{d})$ according to decorated graphs.
A {\em decorated graph} $\Gamma$ consists of the data $(\mathsf{V},\mathsf{E},\mathsf{N},\mathsf{g})$ where
\begin{enumerate}
 \item[(i)] $\mathsf{V}$ is the vertex set,
 \item[(ii)] $\mathsf{E}$ is the edge set (including possible self-edge),
 \item[(iii)] $\mathsf{N} : \{1,2,\dots,n \}\rightarrow \mathsf{V}$ is the marking assignment,
 \item[(iv)] $\mathsf{g}:\mathsf{V}\rightarrow \ZZ_{\ge 0}$ is a genus assignment satisfying
 \begin{align*}
     g=\sum_{v\in\mathsf{V}}\mathsf{g}(v)+h^1(\Gamma).
 \end{align*}
\end{enumerate}
The markings $\mathsf{L}=\{1,\ldots,n\}$ are often called {\em legs}.

Let $\mathsf{G}_{g,n}$ be the set of decorated graph as defined as above. The {\em flags} of $\Gamma$ are the 
half-edges{\footnote{In our convention, markings are not flags.}}. Let $\mathsf{F}$ be the set of flags. For each $\Gamma \in \mathsf{G}_{g,n}$, we choose an ordering 
\begin{align}\label{Ordering}
    \nu_{\Gamma} :\mathsf V\rightarrow \{1,2,\dots,|\mathsf{V}| \}.
\end{align}
Let $\mathsf{G}^{\text{ord}}_{g,n}$ be the set of decorated graph with fixed choice of orderings on vertices.
We will sometimes identify $\mathsf{V}$ with $\{1,2,\dots,|\mathsf{V}|\}$ by \eqref{Ordering}.

\

\subsection{Universal ring}

Denote by 
$$\mathsf{R}$$
the ring generated by 
$$\lambda_{1,j}^{\pm 1}, \lambda_{2,j}^{\pm 1},\dots,\lambda_{n,j}^{\pm 1} ,\ j \in \NN$$
with relations
\begin{align}\label{Relations}
\lambda_{k,j}^2=1 ,\  \text{for} ,\ 1\le k\le n,\ j\in \NN.
\end{align}

\begin{itemize}
 \item 
 A monomial in $\mathsf{R}$ is called {\em canonical} if degree of $\lambda_{i,j}$ are $1$ or $0$ for all $1 \le i \le n$, $j\in \NN $.
 \item 
 {\em The length} of canonical monomial is the number of $\lambda_{i,j}$ whose degree is $1$.
 \item 
 For a monomial $m\in \mathsf{R}$, {\em the length} of $m$ is the length of unique canonical monomial equal to $m$ in $\mathsf{R}$.  
\end{itemize}

\noindent The ring $\mathsf{R}$ will play a fundamental role in the proof of the main theorem.

\begin{Def}
 For monomial $m\in\mathsf{R}$, we say $m$ has type $(\mathsf{a};b)=(a_1,a_2,\dots,a_n;b)\in\NN^{n}\ti\NN$ if $m$ has following form;
 
 \begin{align*}
 m= \prod_j \prod_i \lambda_{i,j}^{a_{i,j}}f
 \end{align*}
where $\sum_j a_{i,j}=a_i$ and $f$ is a monomial whose length is less than or equal to $b$. 

\end{Def}

\begin{Def}
For a polynomial $f\in\mathsf{R}$, we say $f$ has type $(\mathsf{a};b)$ if $f$ is a sum of monomials of type $(\mathsf{a};b)$.
\end{Def}

The following two lemmas are easy to check.

\begin{Lemma}\label{multiplicative}
 If $m_1$ has type $(\mathsf{a}_1;b_1)$ and $m_2$ has type $(\mathsf{a}_2;b_2)$, then $m_1 m_2$ has type $(\mathsf{a}_1+\mathsf{a}_2;b_1+b_2)$.
\end{Lemma}

\begin{Lemma}\label{ev} If $f \in \mathsf{R}$ has type $(\mathsf{a}:b)$ with odd $a_i$ for $0 \le i \le n$ and $b<n$,

 \begin{align*}
 \sum_{\lambda_{i,j}=\pm 1} f = 0.
 \end{align*}
 
\end{Lemma}
\noindent The notation $\sum_{\lambda_{i,j}=\pm 1} f$ in Lemma \ref{ev} denotes the sum of evaluations of $f\in \mathsf{R}$ with all possible choice of $\lambda_{i,j}= 1 \, \text{or} -1.$

\

\subsection{Universal polynomial} 
We review here the definition of universal polynomial $\mathsf{P}$ from \cite{LP}. The universal polynomial play the essential role in calculating higher genus invariants from genus $0$ invariants, see  \cite{Elliptic,Book,LP} for more explanations.
 
 Let $t_0,t_1,t_2, \ldots$ be formal variables. The series
 $$T(c)=t_0+t_1 c+t_2 c^2+\ldots$$   in the
additional variable $c$ plays a basic role. The variable $c$
will later be  replaced by the first Chern class $\psi_i$ of
 a cotangent line  over $\overline{M}_{g,n}$, 
 $$T(\psi_i)= t_0 + t_1\psi_i+ t_2\psi_i^2 +\ldots\, ,$$
 with the index $i$
 depending on the position of the series $T$ in the correlator.

Let $2g-2+n>0$.
For $a_i\in \mathbb{Z}_{\geq 0}$ and  $\gamma \in H^*(\overline{M}_{g,n})$, define the correlator 
\begin{multline*}
    \lann \psi^{a_1},\ldots,\psi^{a_n}\, | \, \gamma\,  \rann_{g,n}= 
    \sum_{k\geq 0} \frac{1}{k!}\int_{\overline{M}_{g,n+k}}
    \gamma \, \psi_1^{a_1}\cdots 
     \psi_n^{a_n}  \prod_{i=1}^k T(\psi_{n+i})\, . 
\end{multline*}

We consider 
$\CC(t_1)[t_2,t_3,...]$
as $\ZZ$-graded ring over $\CC(t_1)$ with 
$$\text{deg}(t_i)=i-1\ \ \text{for $i\geq 2$ .}$$
Define a subring of homogeneous elements by
$$\CC\left[\frac{1}{1-t_1}\right][t_2,t_3,\ldots]_{\text{Hom}} \subset 
\CC(t_1)[t_2,t_3,...]\, .
$$
We easily see 
$$\lann \psi^{a_1},\ldots,\psi^{a_n}\, | \, \gamma \, \rann_{g,n}\, |_{t_0=0}\ \in\
\CC\left[\frac{1}{1-t_1}\right][t_2,t_3,\ldots]_{\text{Hom}}\, .$$
Using the leading terms (of lowest degree in $\frac{1}{(1-t_1)}$), we obtain the
following result.

\begin{Lemma}\label{basis}
The set of genus 0 correlators
 $$
 \Big\{ \, \lann 1,\ldots,1\rann_{0,n}\, |_{t_0=0} \, \Big\}_{n\geq  4} $$ 
freely generate the ring
 $\CC(t_1)[t_2,t_3,...]$ over $\CC(t_1)$.
\end{Lemma}

By  Lemma \ref{basis}, we can find a unique representation of $\lann \psi^{a_1},\ldots,\psi^{a_n}\rann_{g,n}|_{t_0=0}$
in the  variables
\begin{equation}\label{k3k3}
\Big\{\, \lann 1,\ldots,1\rann_{0,n}|_{t_0=0}\, \Big\}_{n\geq 3}\, .
\end{equation}
The $n=3$ correlator is included in the set \eqref{k3k3} to
capture the variable $t_1$.
For example, in $g=1$,
\begin{eqnarray*}
    \lann 1,1\rann_{1,2}|_{t_0=0}&=&\frac{1}{24}
    \left(\frac{\lann 1,1,1,1,1\rann_{0,5}|_{t_0=0}}{\lan 1,1,1\rann_{0,3}|_{t_0=0}}-\frac{\lann 1,1,1,1\rann^2_{0,4}|_{t_0=0}}{\lann 1,1,1\rann^2_{0,3}|_{t_0=0}}\right)\, ,\\
    \lann 1\rann_{1,1}|_{t_0=0}&=&\frac{1}{24}\frac{\lann 1,1,1,1\rann_{0,4}|_{t_0=0}}{\lann 1,1,1\rann_{0,3}|_{t_0=0}}
    \end{eqnarray*}
A more complicated example in $g=2$ is    
\begin{eqnarray*}
\lann \rann_{2,0}|_{t_0=0}&=& \ \ \frac{1}{1152}\frac{\lann 1,1,1,1,1,1\rann_{0,6}|_{t_0=0}}{\lann 1,1,1\rann_{0,3}|_{t_0=0}^2}\\
    & & -\frac{7}{1920}\frac{\lann 1,1,1,1,1\rann_{0,5}|_{t_0=0}\lann 1,1,1,1\rann_{0,4}|_{t_0=0}}{\lann 1,1,1\rann_{0,3}|_{t_0=0}^3}\\& &+\frac{1}{360}\frac{\lann 1,1,1,1\rann_{0,4}|_{t_0=0}^3}{\lann 1,1,1\rann_{0,3}
    |_{t_0=0}^4}\, .
\end{eqnarray*}

\begin{Def} 
For $\gamma \in H^*(\overline{M}_{g,k})$, let $$\pP^{a_1,\ldots,a_n,\gamma}_{g,n}(s_0,s_1,s_2,...)\in \QQ(s_0, s_1,..)$$ be 
the unique rational function satisfying the condition
$$\lann \psi^{a_1},\ldots,\psi^{a_n}\, |\, \gamma\, \rann_{g,n}|_{t_0=0}
=\pP^{a_1,a_2,...,a_n,\gamma}_{g,n}|_{s_i=\lann 1,\ldots,1\rann_{0,i+3}|_{t_0=0}}\, . $$
\end{Def}

We will use following notation.
\begin{multline*}
\ppl
\psi^{k_1}_1, \ldots,\psi^{k_n}_n \, \Big|\, \mathsf{H}_{h}^{p_i}     \ppr_{h,n}^{0+,p_i} 
:=\\
\pP^{k_1,\ldots,k_n,\mathsf{H}_{h}^{p_i}  }_{h,1}\big(\lann 1,1,1\rann_{0,3}^{0+,p_i},\lann 1,1,1,1\rann_{0,4}^{0+,p_i},\ldots \big).
\,
\end{multline*}

\

\section{Higher genus series on $\PP[n]$}
We review here the now standard method used by Givental to express genus $g$ descendent correlators in terms of genus $0$ data, see \cite{SS,KL,Book,LP} for more explanations.

\begin{Def}  Let us define equivariant genus $g$ generating function $F^\mathsf{T}_{g,m}(\mathsf{q}) \in \CC[[\mathsf{q}]](\mathsf{\lambda})$ of $\PP[n]$ as following;

 \begin{align*}
  &F^{\mathsf{T}}_{g,m}[\tau_{k_1}\mathsf{c}_1,\dots,\tau_{k_m}\mathsf{c}_m](\mathsf{q}) := \\
  &\sum_{d_1,d_2,..,d_n \ge 0} q^{d_1}_1 q^{d_2}_2...q^{d_n}_n \int _{ [Q_{g,0}(\PP[n], \mathsf{d})]^{\vir}} \prod^{m}_{i=1}\pi^*(\psi_{i})^{k_i} \text{ev}_{i}^*(\ot_{j=1}^n H_j^{c_{j,i}})
  \in \CC[[\mathsf{q}]](\mathsf{\lambda}). 
 \end{align*}


 \end{Def}

We apply the localization strategy introduced first by Givental for Gromov-Witten theory to the quasimap invariats of $\PP[n]$. We write the localization formula as
$$F^{\mathsf{T}}_{g,m}[\tau_{k_1}\mathsf{c}_1,\dots,\tau_{k_m}\mathsf{c}_m]|_{\lambda_i=1,\overline{\lambda}_i=-1}=\sum_{\lambda_{i,j}=\pm1} \sum_{\Gamma\in\mathsf{G}^{\text{ord}}_{g,0}} \text{Cont}_{\Gamma}\, $$

\noindent where $\text{Cont}_{\Gamma}$ is the contribution to $F^g$ of the $T$-fixed loci associated $\Gamma$. We have the following formula for $\text{Cont}_\Gamma$.
\begin{align*}
 \text{Cont}_\Gamma = \frac{1}{\rm{Aut}(\Gamma)} \sum_{\mathsf{A},\mathsf{B} \in \ZZ_{\ge 1}^{F(\Gamma)}\ti\ZZ_{\ge 1}^{L(\Gamma)}}\prod_{v\in V(\Gamma)} \prod_{e\in E(\Gamma)} \prod_{l\in L(\Gamma)} \text{Cont}_{\Gamma}^{\mathsf{A},\mathsf{B}}(v) \text{Cont}_{\Gamma}^\mathsf{A}(e)\text{Cont}_{\Gamma}^{\mathsf{B}}(l). 
\end{align*}

\noindent We need explanations for the $\text{Cont}_{\Gamma}^{\mathsf{A},\mathsf{B}}(v)$, $\text{Cont}_{\Gamma}^\mathsf{A}(e)$ and $\text{Cont}_{\Gamma}^{\mathsf{B}}(l)$. Let $v \in V(\Gamma)$ be a vertex of genus $\mathsf{g}(v)$.

\begin{align*}
 \text{Cont}_{\Gamma}^{\mathsf{A},\mathsf{B}}(v) = \mathsf{P}[\psi^{a_1-1},\psi^{a_2-1},...,\psi^{a_\alpha-1},\psi^{b_1-1+k_{j_1}},\dots,\psi^{b_\beta+k_{j_\beta}}|\mathsf{H}^{v}_{\mathsf{g}(v)}]^{0+,v}_{\mathsf{g}(v),\alpha+\beta}
\end{align*}

\noindent where, 
\begin{itemize}
\item
$\mathsf{H}^{v}_{\mathsf{g}(v)}=\prod^n_{k=1} \frac{\prod^{\mathsf{g}(v)}_{j=1}\left(\lambda_{k,v}-(-{\lambda}_{k,v})-c_j\right)}{\lambda_{k,v}-(-{\lambda}_{k,v})}$
\item
$(a_1,a_2,\dots,a_\beta)$ are components of $\mathsf{A} \in \ZZ^{F(\Gamma)}_{\ge 1}$ associated to flags of $v$ and $(b_1,b_2,\dots,b_\beta)$ are components of $\mathsf{B} \in \ZZ^{L(\Gamma)}_{\ge 1}$ associated to legs attached to $v$.
\item
$\lambda_{k,v}$ is the weight of $k$-th $\CC^*$ at $\nu(v)$.
\item
$\{c_j | 1 \le j \le \mathsf{g}(v) \}$ are chern roots of  Hodge bundle on $\overline{M}_{\mathsf{g}(v),k}$.
\end{itemize}

\noindent Let $e \in E(\Gamma)$ be an edge between $v_i$ and $v_j$.

\begin{align*}
 \text{Cont}_{\Gamma}^\mathsf{A}(e) = \left[e^{-\frac{\overline{\mathds{U}}^{\PP[n]}_i}{x}-\frac{\overline{\mathds{U}}^{\PP[n]}_j}{y}}e_i\overline{\mathds{V}}^{\PP[n]}_{ij}(x,y)e_j\right]_{x^{a_i-1}y^{a_j-1}}
\end{align*}
where $(a_i,a_j)$ are component of $\mathsf{A} \in \ZZ^{F(\Gamma)}_{\ge 1}$ associated to flags $(v_i,e)$ and $(v_j,e)$. The notation $[\ldots]_{x^{a_1-1}y^{a_2-1}}$
above denotes the coefficient of
 $x^{a_1-1}y^{a_2-1}$ in the series expansion
 of the argument. The bar above $\mathds{U}_i$ and $\mathds{V}_{ij}$ means evaluation $\mathsf{t}=0$. 

\noindent Let $l\in L(\Gamma)$ be the $i$-th leg with insertion $\tau_{k_i}\mathsf{c}_i$.
$$\text{Cont}^{\mathsf{B}}_{\Gamma}(l)=\left[ e^{-\frac{\overline{\mathds{U}}^{\PP[n]}_{v_l}}{x}}\overline{\mathds{S}}^{\PP[n]}_{v_l}(\ot_j H_j^{c_{j,i}}) \right]_{x^{a}}, $$
where $a$ is component of $\mathsf{B}$ associated to leg $l$.
 
\begin{Rmk}
 By \eqref{Ordering}, $\text{Cont}^{\mathsf{A},\mathsf{B}}_{\Gamma}(v)$,$\text{Cont}^\mathsf{A}_{\Gamma}(e)$ and $\text{Cont}^\mathsf{B}_{\Gamma}(l)$ can be considered as elements in $\mathsf{R}[[\mathsf{q},x^{-1},y^{-1}]]$.
\end{Rmk}

\noindent Finally we obtain the following result.
\begin{Prop} For $\PP[n]$, We have
 \begin{align*}
 &F^{\mathsf{T}}_{g,m}[\tau_{k_1}\mathsf{c}_1,\dots,\tau_{k_m}\mathsf{c}_m]|_{\lambda_i=1,\overline{\lambda}_i=-1}  = \\ &\sum_{\lambda_{i,j}=\pm 1}\sum_{\Gamma\in\mathsf{G}^{\rm{ord}}_{g,0}} \frac{1}{\rm{ Aut}(\Gamma)} \sum_{\mathsf{A},\mathsf{B} \in \ZZ_{\ge 1}^{F(\Gamma)}\ti \ZZ_{\ge 1}^{L(\Gamma)}}\prod_{v} \prod_{e}\prod_{l} {\rm Cont}_{\Gamma}^{\mathsf{A},\mathsf{B}}(v) {\rm Cont}_{\Gamma}^\mathsf{A}(e){\rm Cont}_{\Gamma}^\mathsf{B}(l),
 \end{align*}
 where ${\rm Cont}_{\Gamma}^{\mathsf{A},\mathsf{B}}(v)$,${\rm Cont}_{\Gamma}^\mathsf{A}(e)$ and ${\rm Cont}_{\Gamma}^\mathsf{B} (l)$ are defined as above. 
\end{Prop}

\

\section{Proof of main theorem}




\subsection{Overview}
Using \eqref{I} and Proposition \ref{S}, we can explicitly calculate the $\mathds{S}$-operator of $\PP[n]$.
\begin{Prop}\label{SS} For $\{i_1, i_2,..., i_n\}=\{0,1\}$,
  \begin{align*}\mathds{S}^{\PP[n]}(H_1^{i_1}, H_2^{i_2},....,H_n^{i_n})=(z\frac{d}{dt_1})^{i_1}(z\frac{d}{dt_2})^{i_2}...(z\frac{d}{dt_n})^{i_n} \mathds{I}^{\PP[n]}.
  \end{align*}
\end{Prop}

\noindent If we restrict $\mathds{S}$-operator to the fixed points, it admits Birkhoff factorizations. See \cite{CKg0} for more explanations.  
\begin{Def} For fixed point $p \in (\PP[n])^T$, let us define $\mathds{R}_{p,k}^{i_1,i_2,..,i_n}$ by following;\\
 \begin{align*}
 &\mathds{S}^{\PP[n]}(H_1^{i_1} H_2^{i_2} ... H_n^{i_n})|_{p}= \\
 &e^{\frac{\mathds{U}^{\PP[n]}_p}{z}}\lambda_{1,p}^{\delta_{1i_1}}\lambda_{2,p}^{\delta_{1i_2}}...\lambda_{n,p}^{\delta_{1i_n}}(\sum_{k} \mathds{R}_{p,k}^{i_1,i_2,..,i_n}z^k) \in \mathsf{R}[[z,z^{-1},\mathsf{q},\mathsf{t}]](\mathsf{\lambda}).
 \end{align*}
\end{Def}

\noindent Here $\lambda_{k,p}$ is the weight of $k$-th $\CC^*$ at $p$. We consider $\lambda_{k,p}$ as an element of $\mathsf{R}$ by identifying $\lambda_{k,p}$ with $\lambda_{k,1}$.

\begin{Prop}\label{type}
 $\mathds{R}_{p,k}^{i_1,i_2,..,i_n}$ has type $(0,k)$.
\end{Prop}

\begin{proof}
Applying Prop \ref{SS} to \eqref{I}, the proof follows from following lemma.
\end{proof}

\begin{Lemma} We have following assymtotic expansion of local $\mathds{S}$-operators of $\PP^1$ for  fixed point $p \in (\PP^1)^\T$.
 \begin{align*}
   &\mathds{S}^{\PP^1}(1)|_{p}=e^{\frac{\mathds{U}^{\PP^1}_p}{z}}(\sum_k \mathds{Q}_{0k} \lambda_p^{\delta(k)} z^k), \\
   &\mathds{S}^{\PP^1}(H)|_{p}=e^{\frac{\mathds{U}^{\PP^1}_p}{z}}\lambda_p(\sum_k \mathds{Q}_{1k} \lambda_p^{\delta(k)} z^k)
 \end{align*}
 where $\delta(k)=0$ if $k$ is even, $\delta(k)=1$ if $k$ is odd. Here $\mathds{Q}_{ik}$ are constant with respect to $\lambda_p$.
\end{Lemma}

\begin{Rmk} We can calculate $\mathds{Q}_{ik}$ in closed form using Picard-Fuchs equation of $I$-function. Since this is not needed in our paper, we leave the details to the readers.
\end{Rmk}

\begin{Prop}\label{TG} Let $\Gamma\in\mathsf{G}^{\rm{ord}}_{g,0}$ be a decorated graph.
\begin{enumerate}
\item $\text{Cont}_{\Gamma}^\mathsf{A}(v)$ has type $(\overline{\mathsf{g}(v)-1};3\mathsf{g}(v)-3+N_v(\mathsf{A})+N_v(\mathsf{B}))$, where $N_v(\mathsf{A})=\sum_{i=1}^{k}(2-a_i)$ and $N_v(\mathsf{B})=\sum_{i=1}^{\beta}(2-b_i-k_{j_i})$. Here $(a_1,\dots,a_\alpha)$ are components of $\mathsf{A}$ associated to $v$ and $(b_1,\dots,b_\beta)$ are components of $\mathsf{B}$ associated to $v$. The notation $\overline{a}$ for $a\in \NN$ means $(a,\dots,a)\in \NN^n$.

\item $\text{Cont}_\Gamma$ has type $(g-1+\sum_{i}c_{i,1},\dots,g-1+\sum_{i}c_{i,n};3g-3+m-\sum_{i=1}^m k_i)$.
\end{enumerate}
\end{Prop}

\begin{proof} For the first parts, if we expand 
 \begin{align*}
  &\text{Cont}_{\Gamma}^\mathsf{A}(v) = [\psi^{a_1-1},\psi^{a_2-1},...,\psi^{a_\alpha-1},\psi^{b_1-1+k_{j_1}},\dots,\psi^{b_\beta+k_{j_\beta}}|\mathsf{H}^{v}_{\mathsf{g}(v)}]^{0+,v}_{\mathsf{g}(v),\alpha+\beta}\\
  &=\sum_{l_1,l_2,\dots,l_{\mathsf{g}(v)} \ge 0} b_{l_1,l_2,...,l_{\mathsf{g}(v)}} [\psi^{a_1-1},\psi^{a_2-1},...,\psi^{a_\alpha-1},\psi^{b_1-1+k_{j_1}},\dots,\psi^{b_\beta+k_{j_\beta}}|c^{l_1}_1c^{l_2}_2...c^{l_\mathsf{g}(v)}_{\mathsf{g}(v)}]^{0+,v}_{\mathsf{g}(v),\alpha+\beta},
 \end{align*}
 where $(a_1,a_2,..., a_\alpha)$ are components of $\mathsf{A}$ associated to $v$ and $(b_1,\dots,b_\beta)$ are components of $\mathsf{B}$ associated to $v$.
 We can check the followings.
 \begin{itemize}
  \item $b_{l_1,l_2,\dots,l_{\mathsf{g}(v)}}$ has type $(\overline{\mathsf{g}(v)-1};\sum^{\mathsf{g}(v)}_{i=1} l_i)$,
  \item $[\psi^{a_1-1},\psi^{a_2-1},...,\psi^{a_\alpha-1},\psi^{b_1-1+k_{j_1}},\dots,\psi^{b_\beta+k_{j_\beta}}|c^{l_1}_1c^{l_2}_2...c^{l_g}_g]^v_{\mathsf{g}(v),k}$ has type $(\overline{0};3\mathsf{g}(v)_v-3-\sum^{\mathsf{g}(v)}_{i=1} l_i+\sum_{i=1}^{\alpha} (2-a_i)+\sum_{i=1}^{\beta}(2-b_i-k_{j_i}))$.
 \end{itemize}
The proof of first part follows from Lemma \ref{multiplicative}. 
 
For the second parts,  let $e$ be an edge connecting $v_i$ and $v_j$. Then, by Prop \ref{type} and \eqref{WDVV}, we can check that
\begin{align*}
&\text{Cont}_{\Gamma}^\mathsf{A}(e) = \left[(-1)^{k+l}e^{-\frac{\overline{\mathds{U}}^{\PP[n]}_i}{x}-\frac{\overline{\mathds{U}}^{\PP[n]}_j}{y}}e_i\overline{\mathds{V}}^{\PP[n]}_{ij}(x,y)e_j  \right]_{x^{k-1}y^{l-1}}\\
&=\sum_{i=1}^{l}(-1)^{k+l+i+1}\left[e^{-\frac{\overline{\mathds{U}}^{\PP[n]}_i}{x}-\frac{\overline{\mathds{U}}^{\PP[n]}_j}{y}}\sum_{k}\overline{\mathds{S}}^{\PP[n]}(\phi_k)|_{v_i,z=x}\, \overline{\mathds{S}}^{\PP[n]}(\phi^k)|_{v_j,z=y}
\right]_{x^{k-1+i}y^{l-i}}
\end{align*}
has type $(\overline{1};1+(k-1)+(l-1))$.  Here $(k,l)$ are components of $\mathsf{A}$ associated to flags $(v_i,e)$ and $(v_j,e)$. Let $l$ be the $i$-th leg with insertion $\tau_{k_i}\mathsf{c}_i$. Then similarly,
\begin{align*}
    \text{Cont}^{\mathsf{A}}_\Gamma(l)=\left[ e^{-\frac{\overline{\mathds{U}}^{\PP[n]}_{v_l}}{x}}\overline{\mathds{S}}^{\PP[n]}_{v_l}(\ot_j H_j^{c_{j,i}}) \right]_{x^{a}}
\end{align*}
has type $(c_{1,i},c_{2,i},\dots,c_{n,i};a-1)$, where $a$ is the component of $\mathsf{B}$ associated to $l$.
By multiplying all contributions of $\Gamma$, the proof of second part follows from Lemma \ref{multiplicative}.

\end{proof}

\

\subsection{Proof of Theorem \ref{Second}}
Consider the following decomposition of $F^{\mathsf{T}}_{g,m}[\tau_{k_1}\mathsf{c}_1,\dots,\tau_{k_m}\mathsf{c}_m]$.
\begin{align*}
 F^{\mathsf{T}}_{g,m}[\tau_{k_1}\mathsf{c}_1,\dots,\tau_{k_m}\mathsf{c}_m]|_{\lambda_i=1,\overline{\lambda}_i=-1} =\sum_{\lambda_{i,j}=\pm1} \sum_{\Gamma\in\mathsf{G}^{\text{ord}}_{g,0}} \text{Cont}_{\Gamma}.
\end{align*}

\noindent Each $\text{Cont}_{\Gamma}$ has type $(g-1+\sum_{i}c_{i,1},\dots,g-1+\sum_{i}c_{i,n};3g-3+m-\sum_{i=1}^m k_i)$ by Proposition \ref{TG}. Using Lemma \eqref{ev}, we obtain the following result for $n > 3g-3+m-\sum_{i=1}^m k_i$ and odd $g-1+\sum_{i}c_{i,j}$.

 \begin{align*}
 \sum_{\lambda_{i,j}=\pm 1} \text{Cont}_{\Gamma}=0.
 \end{align*}
Therefore we conclude 
\begin{align*}
 F^{\mathsf{T}}_{g,m}[\tau_{k_1}\mathsf{c}_1,\dots,\tau_{k_m}\mathsf{c}_m]|_{\lambda_i=1,\overline{\lambda}_i=-1}=\sum_{\lambda_{i,j}=\pm1} \sum_{\Gamma\in\mathsf{G}^{\text{ord}}_{g,0}} \text{Cont}_{\Gamma}=0,
\end{align*}
for $n > 3g-3+m-\sum_{i=1}^m k_i$ and odd $g-1+\sum_{i}c_{i,j}$.
The $\mathsf{q}$-coefficients in $F^{\mathsf{T}}_{g,m}[\tau_{k_1}\mathsf{c}_1,\dots,\tau_{k_m}\mathsf{c}_m]$ of degree $\mathsf{d}$ with virtual dimension $0$ are independent of $\boldsymbol{\lambda}$. Theroem \ref{Second} is an immediate consequence.

\end{document}